\documentclass[a4paper]{amsart}
\usepackage[english]{babel}
\usepackage[utf8]{inputenc}
\usepackage{lmodern}
\usepackage[T1]{fontenc}

\usepackage{lipsum}
\usepackage{amssymb, amsthm, amsmath,enumerate}
\usepackage[normalem]{ulem}
\usepackage{enumitem}
\usepackage{bbm,dsfont}
\usepackage{xcolor}
\usepackage[all]{xy}
\usepackage{mathtools}

\newcommand{\K}{\mathbb{K}}
\newcommand{\R}{\mathbb{R}}
\newcommand{\C}{\mathbb{C}}
\newcommand{\N}{\mathbb{N}}
\newcommand{\A}{\mathcal{A}}
\newcommand{\U}{\mathcal{U}}

\renewcommand{\L}{\mathcal{L}}
\newcommand{\T}{\mathcal{T}}

\renewcommand{\d}{\, \mathrm{d}}
\renewcommand{\Re}{\operatorname{Re}}

\renewcommand{\mid}{\, \vert \,}

\newcommand{\z}{\zeta}
\renewcommand{\H}{\mathcal{H}}

\newcommand\lra{\longrightarrow}

\theoremstyle{plain}
\newtheorem{theorem}{Theorem}[section]
\newtheorem{proposition}[theorem]{Proposition}
\newtheorem{lemma}[theorem]{Lemma}
\newtheorem{corollary}[theorem]{Corollary}
\theoremstyle{definition}
\newtheorem{definition}[theorem]{Definition}

\newtheorem{example}[theorem]{Example}
\newtheorem{remark}[theorem]{Remark}
\newtheorem{assumption}[theorem]{Assumption}

\begin{document}
\title{Exponential stability for infinite-dimensional non-autonomous port-Hamiltonian Systems
}
\author{Bj\"orn Augner}
\address{Bj\"orn Augner, Technische Universit\"at Darmstadt, Fachbereich Mathematik, Arbeitsgruppe Analysis, Schlossgartenstraße 7,
64289 Darmstadt}
\email{augner@mma.tu-darmstadt.de}
\author{Hafida Laasri
}
\address{Hafida Laasri, Arbeitsgruppe Funktionalanalysis, Fakult\"at f\"ur Mathematik und Naturwissenschaften, Bergische Universit\"at Wuppertal, Gau\ss{}stra\ss{}e 20, 42119 Wuppertal
}
\email{laasri@uni-wuppertal.de}
\thanks{\footnotesize{The second author was supported by Deutsche Forschungsgemeinschaft  DFG (Grant LA 4197/1-1)}}

\maketitle

\begin{abstract}\label{abstract}
We study the non-autonomous version of an infinite-dimensional linear port-Hamiltonian system on an interval $[a, b]$. Employing abstract results on evolution families, we show $C^1$-well-posedness of the corresponding Cauchy problem, and thereby existence and uniqueness of classical solutions for sufficiently regular initial data.
Further, we demonstrate that a dissipation condition in the style of the dissipation condition sufficient for uniform exponential stability in the autonomous case also leads to a uniform exponential decay rate of the energy in this non-autonomous setting.
\end{abstract}

 \medskip


\section{Introduction}\label{Section Introduction}

\subsection{Finite-dimensional, linear port-Hamiltonian control systems}
Hamiltonian mechanics have a very rich history for the modelling of mechanical systems in physics and engineering. At the core of Hamiltonian mechanics lies the notion of a \textit{Hamiltonian} functional, which typically can be interpreteted as an \textit{energy} of such a Hamiltonian system, thereby being a \textit{conserved quantity} in classical Hamiltonian systems.
On the other hand, especially in real-life, engineering applications, mechanical systems cannot be seen as being seperated from the \textit{outside}, but rather interact with their enviroment through certain mechanisms which may include exchange of energy, momentum, heat etc. These mechanisms are summarized under the term \textit{ports}, emphasising that information, energy, momentum etc.\ are transferred from one subsystem into another module of a larger, possibly very complex interconnection structure.
Beginning from the 1960's interest grew in such a \textit{port-based} modelling of physical or engineering systems which may, besides mechanical subsystems, also include e.g.\ thermal subsystems or electrical circuits, all of which can be modelled in the now extended class of port-Hamiltonian systems, i.e.\ dynamical systems which exhibit a (usually energy-driven) Hamiltonian dynamics, possibly also including resistive effects, and are interconnected via suitable ports.
There is a wide range of literature on the topic of (finite-dimensional) port-Hamiltonian modelling and analysis, but only recently \cite{Bea-Meh-Xu-Zwa18} the following class of \textit{non-autonomous} port-Hamiltonian differential-algebraic systems has been investigated:
 \begin{align*}
  E(t) \dot x(t)
   &= [(J - R) Q(t) - E(t) K(t)] x(t) + (B(t)-P(t)) u(t),
   \\
  y(t)
   &= (B(t) + P(t))^T Q(t) x(t) + (S + N) u(t).
 \end{align*}
where e.g.\ $E(t) \in \R^{n \times n}$ corresponds to some time-varying algebraic constraint on the system, $Q(t) \in \R^{n \times n}$ to some time-varying Hamiltonian $H(t)(x) = \frac{1}{2} x^T Q(t)^T E(t) x$, and the other matrices have certain symmetry and anti-symmetry properties, e.g.\ $J$ refers to energy conserving mechanisms whereas $R$ describes resistive effects.
A particular role does the term $K(t)$ play, namely it is chosen in such a way, that the control system has the passivity property
 \[
  \frac{\d}{\d t} \left[ H(t)(x(t)) \right]
   \leq \Re \, (u(t) \mid y(t)),
   \quad
   t \geq 0
 \]
along solutions, if the matrices are chosen properly; see \cite[Definition 5 and Theorem 15]{Bea-Meh-Xu-Zwa18} for details.

\subsection{Infinite-dimensional, linear port-Hamiltonian systems on intervals}
Starting from the 2000's efforts have been made to generalize the theory developed for finite-dimensional port-Hamiltonian systems to the \textit{infinite-dimensional} case; the first step, and in view of string, beams and electrical circuits already practically significant case, being the analysis of linear, infinite-dimensional port-Hamiltonian systems on intervals, i.e.\ systems of PDE's on an interval of the form
 \begin{equation}
  \frac{\partial x}{\partial t}(t, \zeta)
   = \left( P_1 \frac{\partial}{\partial \zeta} + P_0 \right) (\H(\zeta) x(t, \zeta)),
   \quad
   \z \in [a,b], \, t \geq 0
   \label{eqn:PHS}
 \end{equation}
supplemented by initial conditions $x(0,\cdot) = x_0$ and \textit{boundary control and observation} maps $\mathfrak{B}, \mathfrak{C}$ defined via the boundary trace
 \begin{equation}
  \left( \begin{array}{c} \mathfrak{B} x \\ \mathfrak{C} x \end{array} \right)
   = \left[ \begin{array}{c} \tilde W_B \\ \tilde W_C \end{array} \right] \left( \begin{array}{c} (\H x)(t,a) \\ (\H x)(t, b) \end{array} \right).
   \label{eqn:PHS-BC}
 \end{equation}
Here, $x(t,\cdot)$ lies in the state space $X = L^2(a,b; \K^n)$ where $\K$ are either the real $\R$ or complex numbers $\C$.
$\H(\zeta), P_0, P_1$ are $n\times n$ matrices and $\tilde W_B, \tilde W_C \in \K^{n \times 2n}$.
Provided the matrices are suitably chosen, and defining the \textit{energy} of the state $x(t,\cdot) \in X$ as
 \[
  H(t)
   = \frac{1}{2} \int_a^b (x(t, \zeta) \mid \H(\zeta) x(t, \zeta)) \d \zeta
 \]
one can consider impedance-passive systems, i.e.\ systems for which the (classical) solutions $x$ to the PDE \eqref{eqn:PHS} satisfy
 \[
  \frac{\d}{\d t} H(t)
   \leq \Re \, (u(t) \mid y(t)),
   \quad
   t \geq 0
 \]
where $(u(t), y(t)) := (\mathfrak{B} x(t), \mathfrak{C} x(t))$.
Autonomous port-Hamiltonian systems, i.e.\ systems of the form \eqref{eqn:PHS}--\eqref{eqn:PHS-BC}, have
 been investigated recently, e.g.\ in \cite{G-Z-M,Jac-Zwa12,Aug-Jac14,Vi,Z-G-M-V, Vi-Go-Z-Scharf}.
 Well-posedness and uniform exponential (or, asymptotic) stability for autonomous port-Hamiltonian systems can in most cases be tested via a simple matrix condition \cite{Jac-Zwa12, Jac-Mor-Zwa15}. If the \textit{Hamiltonian density} $\H$ is coercive as a matrix multiplication operator on $L^2(a,b;\K^{n\times n}),$ the energy of the system  \[\frac{1}{2} \|x\|_\mathcal H^2:=\frac{1}{2} (x\mid x)_\H:= \frac{1}{2} \displaystyle\int_a^b (x(\zeta) \mid \H(\zeta) x(\zeta)) \d\zeta\] 
 defines an equivalent (to the usual $L^2$-norm) norm on $L^2(a,b;\mathbb C^n)$. In this case, the existence of classical and mild solutions with non-increasing energy can be tested via a simple matrix condition. More precisely, consider the linear operator 
 \begin{align*}A\H&=P_1\frac{\partial}{\partial\zeta}\H+P_0\H\\
 D(A\H)&=\Big\{x\in L^2(a,b;\K^n)\mid \H x\in H^1(a,b;\K^n), \tilde W_B\left[\begin{array}{c}
 (\H x)(t,b) \\
 (\H x)(t,a) \\
 \end{array}\right]=0 \Big\}.
 \end{align*}
 Then $A\H$ generates a contractive $C_0$-semigroup on the energy state space $X_\mathcal H:=L^2_{\mathcal H}(a,b;\K^n):=(L^2(a,b;\K^n),\|\cdot\|_\mathcal H)$ if and only if $\label{matrix formal} W_B\Sigma W_B^*\geq 0$ and $P_0 + P_0^* \leq 0$, where  \[W_B:=\tilde W_B\begin{bmatrix}
  P_1& -P_1\\
  I& I 
  \end{bmatrix}^{-1}\text{ and } \Sigma:=
  \begin{bmatrix}
  0& I\\
  I&0 
  \end{bmatrix},\]
 and this is exactly the case, if the operator $A$ is dissipative on $L^2(a,b;\K^{n \times n})$ (equipped with the standard $L^2$-norm).
Moreover, the asymptotic behaviour of the solution of (\ref{Nport-Hamiltonian system})-(\ref{NBoundary control})
has also been studied, e.g.\ in \cite{Vi-Zw-Go-M, Jac-Zwa12, Aug-Jac14}. The authors give a result on exponential stability using a Lyapunov method in \cite{Vi-Zw-Go-M,Jac-Zwa12} and  using a frequency domain method in \cite{Aug-Jac14} based on classical stability theorems by Gearhart, Pr\"uss and Huang, and by Arendt, Batty, Lyubich and V\~u. It has been proved in \cite[Theorem III.2]{Vi-Zw-Go-M} that the operator $A\H$ from above generates an exponentially stable $C_0$-semigroup if for some constant $c>0$ one of the following conditions is satisfied for all $x\in D(A\H).$
\begin{align}
\label{StabiCondAut1}\Re (A\H x\mid x)_\H&\leq -c |\H(b)x(b)|^2.
\\ \label{StabiCondAut2} \Re (A\H x\mid x)_\H&\leq -c |\H(a)x(a)|^2.
\end{align}
These inequalities hold, for example, if $W_B \Sigma W_B^*> 0$ \cite[Lemma 9.4.1]{Jac-Zwa12}, i.e.\ $W_B \Sigma W_B^*$ is a symmetric, positive definite matrix.
 Port-Hamiltonian systems of order $N\geq 2$ also have been investigated with similar results in \cite{Aug-Jac14, Vi}.
 
\noindent In contrast to autonomous port-Hamiltonian systems, the non-autonomous and infinite-dimensional state space situation with $\H$ and/or $\tilde W_B$ depending on the time variable has not been considered so far. The main purpose of this paper, therefore, is to generalize the known results on well-posedness and exponential stability for autonomous port-Hamiltonian systems to the non-autonomous setting. In particular, we show that the technique used in \cite{Vi-Zw-Go-M, Jac-Zwa12} for the proof of uniform exponential stability can be applied to non-autonomous port-Hamiltonian systems as well.
To be more precise, we consider the following non-autonomous partial differential equation
\begin{align}\label{Nport-Hamiltonian system}
\frac{\partial x}{\partial t} (t,\zeta)&=\Big(P_1\frac{\partial}{\partial\zeta}+P_0\Big)(\mathcal{H}(t,\zeta)x(t,\zeta)) + K(t,\zeta) x(t,\zeta)
 &&\zeta\in[a,b],\  t\geq 0,
\\ \label{Nport-Hamiltonian system-initial}x(0,\zeta)&=x_0(\zeta),
 &&\zeta\in [a,b],
\end{align}
on the state space $X=L^2(a,b;\K^n)$ (for $\K = \R$ or $\C$), and with boundary conditions of the type
\begin{align} \label{NBoundary control}\tilde W_B\left[
\begin{array}{c}
(\mathcal Hx)(t,a) \\
(\mathcal Hx)(t,b) \\
\end{array}
\right]&=0, \qquad \qquad\quad  t\geq 0,
\end{align}
where $x(t,\zeta)$ takes its values in $\mathbb K^n, \mathcal H(t,\zeta), P_0,P_1$ are $n\times n$ matrices and
$\tilde W_B$ is an $n\times 2n$ matrix, but $\H(t,\zeta)$ may depend on the time variable $t \geq 0$. Moreover, in contast to the autonomous case a further additive perturbation $K(t,\zeta)$ may be present where $K(t,\zeta) \in \K^{n \times n}$ is a matrix, possibly depending on $t \geq 0$ and $\z \in [a,b]$.
We call this class of PDE an (infinite-dimensional, linear) \textit{non-autonomous port-Hamiltonian system} and it covers, among others, the wave equation, the transport equation,  beam equations as well as certain models of electrical circuits, all with possibly time- and spatial dependent parameters.

\noindent To do so, we write system (\ref{Nport-Hamiltonian system})-(\ref{NBoundary control}) as an abstract non-autonomous evolution equation of the form 
 \begin{align}\label{IntNonautonomous InitialAbstValueProblem}
 \dot x(t)-AB(t)x(t)-P(t)x(t)&=0\  \ {\rm   a.e.\  on  } \ [0,\infty),\\ \label{IntNonautonomous InitialAbstValueProblemInitialCdd} x(0)&=x_0,\quad 0>0,
 \end{align}
where $A: D(A):X\lra X$ is the generator of a contractive $C_0$-semigroup, $B:[0,+\infty)\lra \L(X)$ is a time-dependent multiplicative perturbation and $P: [0, +\infty)\lra \L(X)$ is time-dependent additive perturbation. More precisely, we investigate whether the operator family $\mathcal{A} = \{A B(t) -P(t): t \in [0, \infty) \}$ generates a strongly continuous evolution family $\mathcal{U} = (U(t,s))_{t \geq s \geq 0}$ on the state space $X$. The well-posedness of this abstract class has been studied  by Schnaubelt and Weiss \cite{SchWei10}. The parabolic case has been investigated in \cite{AuJaLa15} with $P(t) = 0$. 
\noindent This paper is organized as follows. In Section \ref{Section1} we recall some abstract results on the theory of evolution families and the well-posedness for non-autonomous evolution equations and give some preliminary results. In Section \ref{Section 2} we provide sufficient conditions for which the non-autonomous port-Hamiltonian system is well-posed and the corresponding evolution family is exponentially stable. We also discuss the necessecity of some imposed parameter restrictions. The last section is devoted to the examples of the one-dimensional wave equation and the Timoshenko beam model.

\section{Background on evolution families and preliminary results}\label{Section1}
Throughout this section, $X$ is a Hilbert space over $\K=\C$ or $\R.$ We denote by $( \cdot|\cdot)$ the scalar product and by $\|\cdot\|$ the norm on $X.$ Let $\{A(t) \mid t\geq 0\}$ be a family of linear, closed operators with domains $\{D(A(t))\mid t\geq 0\}.$ Consider the \textit{non-autonomous Cauchy problems} 
\begin{equation}\label{Nonautonomous InitialValueProblem}
\dot u(t)-A(t)u(t)=0\  \ {\rm   a.e.\  on  } \ [s,\infty),\ u(s)=x_s, (s>0).
\end{equation}
Recall that a continuous function $u:[s,\infty)\lra X$ is called a \textit{classical solution} of (\ref{Nonautonomous InitialValueProblem}) if $u(t)\in D(A(t))$ for all $t\geq s, u\in C^1((s,\infty), X)$ and $u$ satisfies (\ref{Nonautonomous InitialValueProblem}), so that in particular $A u \in C((s,\infty),X)$. 

As in the autonomous case, well-posedness means  that for sufficiently regular initial data, \eqref{Nonautonomous InitialValueProblem} has a unique classical solution which continuously depends on the initial data $x_s \in D(A(s))$.
\par\noindent The study of non-autonomous evolution equations has a long history which goes back to Vito Volterra in 1938 \cite{V.H}. However, it was only in 1950--1970 that a general
 theory has been developed by T. Kato \cite{K}, \cite{Ka70},
 H.~Tanabe \cite{Ta79}, P. E. Sobolevsky \cite{So} and others. P. Acquistapace
 and B. Terreni \cite{Ac-Te}
 extended the previous work by Kato, Tanabe, Sobolevsky
 and obtained some of the most powerful results. Their approach is based on the
 time-discretization of the given equation and the use of
 semigroup theory.  Another approach to these equations using semigroup
 theory and evolution families has been used by J. S. Howland \cite{How}. This approach
 is presented in the monograph \cite{CL} by C. Chicone and Y. Latushkin, and has been further developed by R. Nagel and G. Nickel \cite{Na-Ni}, R. Schnaubelt
 \cite{Sch2} and many other authors.  For the Hilbert space setting, a variational approach has been developed essentially by Lions's school, leading to the existence and uniqueness of weak solutions \cite{Lio61,Lio-Mag72}.
\medskip

\noindent The existence and uniqueness for solutions of a non-autonomous Cauchy problem is closely related to the existence of a \textit{(strongly continuous) evolution family} 
 \[\U:=\left\{ U(t,s): t\geq s\geq 0 \right\}\subset \L(X)\]
 i.e., a family that has the following properties: $U(t,t)=I$ and $U(t,s)= U(t,r)U(r,s)$ for every $0\leq s\leq r\leq t$ and  $U(\cdot,\cdot): \Delta\lra \L(X)$ is strongly continuous where $\Delta:=\{(t,s) \in \R^2 \mid t \geq s \geq 0\}.$
 More precisely, if the abstract Cauchy problem is well-posed for all initial data $x_s \in D(A(s))$ and initial times $s \geq 0$, i.e.\ \eqref{Nonautonomous InitialValueProblem}
has a unique classical solution which depends continuously on the initial data, then the solutions $x(t,s,x_s)$ define an evolution family $\U \subset \L(X)$ given by $U(t,s) x_s := x(t,s,x_s)$.
On the other hand, for an evolution family to be the solution operator (for classical solutions) of an abstract Cauchy problem, $\U$ needs to satisfy further properties then just being an evolution family, e.g.\ $U(\cdot,s)x \in C^1((s,\infty);X)$ for all $x_s \in D_s$ for some dense subsets $D_s \subset X$ \cite[IV.8]{EnNa00}.

\noindent The \textit{(exponential) growth bound} of an evolution family $\U$ is defined by 
 \[w_0(\U):=\inf\left\{w\in \R\mid \text{ there is } M_w\geq 1 \text{ with } \|U(t,s)\|\leq M_we^{w(t-s)} \text{ for } t\geq s\right\}.\]
 The evolution family is called \textit{exponentially bounded} if $w_0(\U)<+\infty$ and \textit{ exponentially stable} if $w_0(\U)<0.$ 
 If $(T(t)))_{t \geq 0}$ is a $C_0$-semigroup on $X$ then $U(t,s):=T(t-s)$ yields a strongly continuous (and exponentially bounded) evolution family. In contrast to $C_0$-semigroups which are always exponentially bounded, i.e.\ $\| T(t)\| \leq M e^{\omega t}$ ($t \geq 0$) for some $M \geq 1$, $\omega \in \R$, see e.g.\ \cite[Proposition I.5.5]{EnNa00}, the same cannot be said about evolution families in general. Moreover, in many cases uniform exponential stability for a $C_0$-semigroup (or, the growth bound) can be determined via the spectrum of its generator, e.g.\ for analytic semigroups. In contrast, for evolution families this fails to be true even in the finite dimensional case, see e.g.\ \cite[Example VI.9.9]{EnNa00}. Nevertheless, the asymptotic behaviour of an exponential evolution family can be characterized in terms of the associated \textit{evolution semigroup}. Indeed, it is well known \cite[Section 3.3]{CL} that to each exponentially bounded evolution family $\U$ one may associate a unique $C_0$-semigroup $\T$ on $L^p([0,+\infty);X)$ ($p\in[1,\infty)$) defined for each $f\in L^p([0,+\infty);X)$ by setting
 
 \[ (\T(t)f)(s):=\left\{%
 \begin{array}{ll}
 U(s,s-t)f(s-t)\quad &\text{for } s, s-t\in[0,+\infty),\\
 0 & \text{for } s\in [0,+\infty), t-s\notin [0,+\infty), \\
 \end{array}%
 \right. \]
Denoting by $G$ the generator of $\T$, the following characterisation is well known: Let $\U$ be an exponentially bounded evolution family on $X$ and let $p\in[1,\infty).$ Then the following assertions are equivalent.
\begin{itemize}
	\item [$(i)$] $\U$ is exponentially stable. 
	\item [$(ii)$] The generator $G$ of the associated evolution semigroup is surjective. 
	\item [$(iii)$] For all $x\in X$ and $s\geq 0$ there exists a constant $M>0$ such that  
	\[\int_s^\infty\|U(t,s)x\|^p \d t \leq M \|x\|^p. \]
	\item [$(iv)$] For all $f\in L^p([0,+\infty);X)$ one has $U\star f\in L^p([0,+\infty);X).$
\end{itemize}

\noindent For the proof and other concepts of stability, we refer e.g., to \cite{Sch2, BaChTo02, Ha82, PauSe18} and the references therein. 
\par

\noindent We abstain from this route towards exponential stability for non-autonomous port-Hamiltonian systems and follow a different approach for the study of exponential stability by mimicking the techniques used in \cite{Vi-Zw-Go-M, Jac-Zwa12} for the autonomous case. These are based on an idea of Cox and Zuazua \cite{Cox-Zua95}.

\begin{definition}\label{Definition well posedness} $(a)$ The non-autonomous Cauchy problem (\ref{Nonautonomous InitialValueProblem}) is called $C^1$-\textit{well posed} if there is a family $\{Y_t\mid  t\geq 0\}$ of dense subspaces of $X$ such that:
\begin{itemize}
	\item [$(i)$] $Y_t\subseteq D(A(t))$ for all $t\geq 0.$
	\item [$(ii)$] For each $s\geq 0$ and $x_s\in Y_s$ the Cauchy problem \eqref{Nonautonomous InitialValueProblem} has a unique classical solution $u(\cdot,s,x_s)$ with $u(t,s,x_s)\in Y_t$ for all $t\geq s.$
	\item [$(iii)$] The solutions depend continuously on the initial data $s, x_s.$
\end{itemize} 
In this case we also say that (\ref{Nonautonomous InitialValueProblem}) is  $C^1$-\textit{well posed} on $Y_t$ if we want to specify the \textit{regularity subspaces} $Y_t, t\geq 0.$
\par\noindent $(b)$ We say that the family $\left\{A(t), t\geq 0 \right\}$ \textit{generates an evolution family} $\U$ if there is a family $\{Y_t\mid  t\geq 0\}$ of dense subspaces of $X$ with $Y_t\subset D(A(t)) ,\ U(t,s)Y_s\subset Y_t$ and for every $x_s\in Y_s$ the function $U(\cdot,s)x_s$ is a classical solution of (\ref{Nonautonomous InitialValueProblem}).
\end {definition}

\noindent In the autonomous situation, i.e., if $A(t)=A$ is a time-independent operator, it is well known that the associated Cauchy problem is $C^1$-well-posed if and only if $A$ generates a $C_0$-semigroup $(T(t))_{t\geq 0}$. In this case the unique classical solution to (\ref{Nonautonomous InitialValueProblem}) is given by $T(\cdot-s) x_s$ for each $x_s\in D(A)$, $s \geq 0$.
\par\noindent 

	
\par\noindent As mentioned in the introduction, the evolution law $(i)$ does not guarantee that the evolution family is strongly differentiable in the first component and that $\U$ is generated by a family of linear closed operators. In fact, it may even happen that the trajectory $U(\cdot, s)x$ is differentiable only for $x=0.$ The standard counterexample is given by $U(t,s)=\frac{p(t)}{p(s)}$ with $X=\C$ and $p$ is a nowhere differentiable function such that $p$ and $1/p\in C_b(\R).$ On the other hand, the following characterization holds:  
\begin{proposition}The Cauchy problem  (\ref{Nonautonomous InitialValueProblem}) is $C^1$-\textit{well posed} if and only if   $\left\{A(t), t\geq 0 \right\}$ generates a unique evolution family.
\end{proposition}
For this statement, we refer to \cite[Proposition 9.3]{EnNa00} or \cite[Proposition 3.10]{GN96}. 
\medskip

Let us for the moment consider the special case where the domains $D(A(t))=D$ are time independent. Then  it is well known that the family of closed, linear operators $\{A(t) \mid t\geq 0\}$ generates a unique evolution family with $Y_t=D$ for all $t\geq 0$ if  the following assumptions are satisfied:
\medskip

({\bf H1}) $\{A(t) \mid t\geq 0\}$ is \textit{Kato-stable}: i.e., for $T\geq 0$ there are $\omega\in \R , M\geq 1$ such that 
\begin{equation*}\label{eq Kato Stability} 
\|(\lambda-A(t_n))^{-1}(\lambda-A(t_{n-1}))^{-1}\cdots(\lambda-A(t_0))^{-1}\|_{\L(X)}\leq M(\lambda-\omega)^{-n},
\end{equation*}
for all $0\leq t_0\leq t_1\leq \cdots\leq t_n\leq T, n\in \N$ and all  $\lambda>\omega.$
\par ({\bf H2}) For each $T\geq 0$ and $x\in D$ the function $A(\cdot)x\in C^{1}([0,T]; X).$

\medskip

This result is due to Kato \cite{Ka70}, we refer to the survey paper \cite{Sch2} for further reading. The stability condition $({\bf H1})$ is always fulfilled if for each $t\geq 0$ the operator $A(t)$ generates a contractive $C_0$-semigroup. For general semigroups, we recall the following two useful stability tests \cite[Propositions 4.3.2 and 4.3.3]{Ta79}:

\begin{proposition}\label{Proposition Kato stability test}Assume that there is a family $\left\{\|\cdot \|_t, t\geq 0\right\}$ of norms on $X$ that are equivalent to the original time independent norm $\|\cdot\|$, and such that  for each $T\geq 0$ there exists a constant $c\geq 0$ such that 
\begin{equation}\label{stability test condition}
\|x\|_t\leq e^{c|t-s|}\|x\|_s
\end{equation}
for all $x\in X$ and $t,s\in[0,T].$
If $A(t)$ generates a contractive semigroup  on $X_t:=(X, \|\cdot\|_t)$ for all $t\geq 0$ then the family $\left\{A(t), t\geq 0 \right\}$ is Kato-stable. If, moreover, $M:[0,\infty)\lra \L(X)$ is a locally bounded function, then the perturbed family $\{ A(t)+M(t) \mid t\geq 0 \}$ is again Kato-stable.
\end{proposition}

\subsection{Time-varying multiplicative perturbations of contraction semigroup generators}
In this subsection we consider the special case where $A(t)$ as in \eqref{Nonautonomous InitialValueProblem} is defined as a bounded non-autonomous multiplicative perturbation of a dissipative operator. More precisely, let $B:[0,\infty)\lra \L(X)$  be a function of class $C^2.$  Assume that $B$ is self-adjoint and uniformly coercive, i.e., $B(t)^*=B(t)$ and 
\begin{equation}\label{uniformly coerciveness of B}
 (B(t)x|x)\geq \beta\|x\|^2 
\end{equation}
for some constant $\beta>0$, and for all $t\geq 0$ and $x\in H.$ Then for each $t\in[0,\infty)$ the function 
\begin{equation}\label{Eq Equivalent norm}\|x\|_{t}:=\sqrt{(B(t)x|x)}=\|B^{1/2}(t)x\|
\end{equation}
defines a norm which is equivalent to the time-independent reference norm $\|\cdot\|.$ Moreover, the norms $\left\{\|\cdot\|_t, t\geq 0\right\}$ are uniformly equivalent to $\|\cdot\|$ on each compact interval $[0,T]\subset [0,\infty).$ Indeed, let $T>0$ and set $\beta_T:=\max_{t\in[0,T]}\|B(t)\|_{\L(X)}$, then we have 
\begin{equation}\label{eq LocEquiNorm}
\frac{1}{\beta_T}\|x\|_s\leq \|x\|\leq \frac{1}{\beta}\|x\|_t\ \text{ for all } t,s\in[0,T]\ \text{ and } \ x\in X.
\end{equation}
\medskip

\par\noindent Let $A:D(A)\subset X\lra X$ be the infinitesimal generator of a contractive $C_0$-semigroup on $X$ and consider the following class of non-autonomous problems:
\begin{equation}\label{Nonautonomous InitialAbstValueProblem}
\dot x(t)-AB(t)x(t)=0\  \ {\rm   a.e.\  on  } \ [s,\infty),\ x(s)=x_s, \, s\geq 0.
\end{equation}
Here, the operators $AB(t)$ are defined on their natural domains  
\[D(AB(t))=\{ x\in X\mid B(t)x\in D(A)\}\]
which (in contrast to $D(A)$) may depend on $t.$ Then for each $t\geq 0$ the operator $AB(t)$, and thus by similarity $B(t)A,$ generates a contractive semigroup on $X_t$ \cite[Lemma 7.2.3]{Jac-Zwa12}.  Note that $AB(t)$ and $B(t)A$ are similar since $B(t)AB(t)B^{-1}(t)=B(t)A$ and  $B^{-1}(t)\in\L(X)$ for every $t\geq 0.$

Theorem \ref{main Thm1} below shows that \eqref{Nonautonomous InitialAbstValueProblem} is  $C^1$-well posedness and that the associated evolution family is exponentially bounded with $M_\omega = 1$. The latter will be needed in the next section where we study the exponential stability of the evolution family generated by non-autonomous port-Hamiltonian systems. We point out that the $C^1$-well posedness for non-autonomous evolution equations of the form (\ref{Nonautonomous InitialAbstValueProblem}) has been studied by Schnaubelt and Weiss \cite[Proposition 2.8-(a)]{SchWei10}.
We will not prove this theorem in full detail, but for sake of making this paper easier readable sketch the ideas used in the proof. In fact, the proof of \cite[Proposition 2.8-(a)]{SchWei10} is based on an perturbation argument due to Curtain and Pritchard, see \cite[Proposition 2.7]{SchWei10}, which is not needed here. 
\begin{theorem}\label{main Thm1} The Cauchy problem \eqref{Nonautonomous InitialAbstValueProblem} is $C^1$-well posed with regularity space $\left\{D(AB(t)), t\geq 0\right\}.$ Further, for each compact interval $[0,T]\subset [0,\infty)$  there exist constants $c_T, \kappa_T>0$ such that 
\begin{equation}\label{ExpStabIneq}
\|U(t,\tau)x\|_t^2\leq e^{\frac{M_T}{\beta}(t-s)}\|U(s,\tau)x\|_s^2\quad (0\leq \tau\leq s\leq t\leq T),
\end{equation}
for each $x\in X$ where $M_T=\underset{t\in[0,T]}{\max}\|\dot B(t)\|_{\L(X)}.$
\end{theorem}
\begin{proof}[Sketch of proof]
 \begin{enumerate}
  \item
   First, one shows that the operator family $\mathcal{A} = \{A B(t) \mid t \geq 0\}$ generates an evolution family $\mathcal{U}$ with regularity spaces $Y_t = D(A B(t))$, $t \geq 0$, if and only if $\tilde {\mathcal{A}} = \{ B(t) A + \dot B(t) B(t)^{-1} \mid t \geq 0\}$ generates an evolution family $\mathcal{V}$ with time-independent regularity space $D(A)$, and that
    \[
     U(t,s)
      = B(t)^{-1} V(t,s) B(s),
      \quad 0 \leq s \leq t.
    \]
   This result is based on the fact that $B \in C^1([0, \infty), \L(X))$ is uniformly coercive.
  \item
   Being able to switch to the case of a time-independent regularity space $D(A)$ for $\tilde {\mathcal{A}}$, one next proves that for $B \in C^2([0,\infty), \L(X))$ Kato's conditions ({\bf H1})--({\bf H2}) hold true, then uses Proposition \ref{Proposition Kato stability test}.
  \item
   Finally, inequality \eqref{ExpStabIneq} can be proved by considering
   \begin{align}
    \frac{\d}{\d t}\|U(t,\tau)x_\tau\|^2_t
    &\label{locBddEVF}
    \leq (\dot B(t)U(t,\tau)x_\tau\mid U(t,\tau)x_\tau)
    \\
    \nonumber
    &\leq \frac{M_T}{\beta}\|U(t,\tau)x_\tau\|_t^2,
    \qquad \tau\leq t\leq T,
   \end{align}
   and using Gronwall's inequality.
 \end{enumerate}
\end{proof}
The inequality \eqref{ExpStabIneq} will be essential in the next section where we study the exponential stability of the evolution family $\U$ generated by non-autonomous port-Hamiltonian systems.
\begin{remark}$(i)$\ Under the assumption of Theorem \ref{main Thm1} it is easy to see that the evolution family $\U$ generated by ${\A}:=\{A B(t) \mid t\geq 0\}$ is locally exponentially bounded. In fact, taking $s=\tau$ in \eqref {ExpStabIneq} and using \eqref{eq LocEquiNorm} we obtain that 
\begin{align*}\|U(t,s)x\|^2&\leq \frac {1}{\beta} \|U(t,s)x\|_t^2 \leq  \frac {1}{\beta} e^{\frac{M_T}{\beta}(t-s)}\|x\|_s^2
\leq \frac {\beta_T}{\beta} e^{\frac{M_T}{\beta}(t-s)}\|x\|^2, \quad x \in X
\end{align*}
for all $T>0$ and each $0\leq s\leq t\leq T.$ Recall that all  strongly continuous semigroups are exponentially bounded. This is, however, not the case for general evolution families, cf. [EN00, Section VI.9].
\medskip
\par\noindent $(ii)$\  If in addition $\dot B(t)\leq 0$ (in the sense that $\dot B(t) = \dot B(t)^\ast$ and $(\dot B(t) x \mid x) \leq 0$ for all $x \in X$) for all $t\geq 0$, then \eqref{locBddEVF}
implies that $t\mapsto \|B^{1/2}(t)U(t,s)x\|$ is decreasing on $[s,\infty)$ for each $s\geq 0$ and $x\in X.$ This can be seen as a generalization of \cite[Lemma 7.2.3]{Jac-Zwa12} to the non-autonomous setting. 
\end{remark}
 \begin{corollary}
 \label{cor:passive_phs}
  In the situation of Theorem \ref{main Thm1}, consider an additional additive term $K \in C^1(\R_+; \L(X))$, and the corresponding non-autonomous Cauchy problems
   \begin{equation}
    \dot x(t) - A B(t) x(t) - K(t) x(t)
     = 0
     \quad \text{a.e.\ on } [s, \infty),
     \quad
     x(s) = x_s, \, (s \geq 0).
     \label{eqn:NACP}
   \end{equation}
  These are $C^1$-wellposed, and their respective solutions are given by an evolution family $\mathcal{U} = (U(t,s))_{0 \leq s \leq t}$.
  Moreover, if $K$ is such that
   \[
    B(t) K(t) + K(t)^\ast B(t) + \dot B(t) \leq 0
   \]
  is negative semi-definite for all $t \geq 0$, then
   \[
    \|U(t,\tau) x\|_t
     \leq \|U(s,\tau) x\|_s,
     \quad
     \text{for all } 0 \leq \tau \leq s \leq t, \, x \in X.
   \]
  Note that the latter inequality holds true, in particular for $K = 0$ and $\dot B(t) \leq 0$.
 \end{corollary}
 \begin{proof}
  In view of the perturbation result Proposition \ref{Proposition Kato stability test}, the well-posedness and existence of an evolution family follows by a similiar argument as in the proof of Theorem \ref{main Thm1}.
  For the contraction property, one calculates for any $0 \leq \tau \leq s \leq t$ and $x \in D(A B(\tau))$ that
   \begin{align*}
    \frac{\d}{\d t} \|U(t,\tau) x\|_t^2
     &= (\frac{\partial}{\partial t} U(t,\tau) x \mid B(t) U(t,\tau) x)
      + (U(t,\tau) x \mid B(t) \frac{\partial}{\partial t} U(t,\tau) x)
      + (U(t,\tau) x \mid \dot B(t) U(t,\tau) x)
      \\
     &= 2 \Re \, (A U(t,s) x \mid U(t,s) x)
      + 2 \Re (B(t) K(t) U(t,s) x \mid U(t,s) x)
      + (\dot B(t) U(t,s) x \mid U(t,s) x)
      \\
     &\leq ([B(t) K(t) + K(t)^\ast B(t) + \dot B(t)] U(t,s) x \mid U(t,s) x)
   \end{align*}
  by dissipativity of the operator $A$.
  If $B(t) K(t) + K(t)^\ast B(t) + \dot B(t) \leq 0$, the right-hand side is non-positive, and the assertion follows by density of $D(A B(\tau))$ in $X$.
 \end{proof}
%
\section{Non-autonomous port-Hamiltonian systems}\label{Section 2}
In this section, we are concerned with the linear non-autonomous port-Hamiltonian system  \eqref{Nport-Hamiltonian system}-\eqref{NBoundary control} introduced in Section \ref{Section Introduction}.  Recall that in this case we have $X=L^2(a,b;\K^n)$ with the standard $L_2$-inner product. Throughout this section, we always assume the following:
\begin{assumption}\label{Assumption1}\end{assumption}
\begin{itemize}
	\item  [$(i)$]$\Re P_0 := \frac{P_0 + P_0^*}{2} \leq 0$ is negative semi-definite.
	\item  [$(ii)$] $P_1$ is invertible and self-adjoint.
	\item  [$(iii)$] $\tilde W_B \Sigma \tilde W_B^*\geq 0$ and $\operatorname{rank} \ \tilde W_B=n$
 \item  [$(iv)$] $\H\in C^2([0,\infty);L^\infty(a,b;\K^{n\times n}))$ and there exist $m, M\geq 0$ such that
 \begin{equation}\label{coercivity H}
 m \leq \H(t,\xi)=\H^*(t,\xi)\leq M, \quad {\rm a.e. }\ \xi\in [a,b], t\geq 0.
 \end{equation}
   \item[$(v)$] $K \in C^1([0,\infty); L^\infty(a,b; \K^{n \times n}))$
\end{itemize}
As in the introduction, we consider the (unperturbed) operator $A: D(A) \subseteq X \rightarrow X$ defined by
 \begin{equation}
  A x
   = P_1 \frac{\partial}{\partial \z} x + P_0 x,
   \quad
   x \in D(A) = \{x \in H^1(a,b; \K^n): \, \tilde W_B \left( \begin{smallmatrix} x(a) \\ x(b) \end{smallmatrix} \right) = 0 \}
   \label{def:A}
 \end{equation}
and its multiplicative perturbations $A \H(t)$ with time-varying domain $D(A \H(t)) = \{x \in X_{\H(t)}: \, \H(t) x \in D(A)\}$.
\begin{remark}
\begin{itemize}
  \item[(i)]
   The constant matrix $P_0$ can be replaced by a matrix valued funtion \[ P_0 \in C([0,\infty); L^\infty(a,b;\K^{n \times n})) \] and without any restriction on the dissipativity of $P_0$.
   In this case the operator $A(t)$ may depend on the time variable $t$, yet its domain is independent of $t$ as for every fixed $t \in [0,T]$, the matrix valued function $P_0(t,\cdot)\in \L(X)$ is just a bounded perturbation of the corresponding operator for $P_0 = 0$.
   However, $A$ does not generate a contractive $C_0$-semigroup on $L^2(a,b;\K^{n \times n})$ unless the symmetric part $\Re P_0(t,\cdot) \leq 0$ is negative-semidefinite a.e.\,; cf.\, \cite{Aug-Jac14}.
   On the other hand, w.l.o.g.\ time-varying parts of $P_0$ may be absorbed into the time-varying perturbation $K(t,\zeta)$, so w.l.o.g.\ one may even assume that $P_0 = 0$.
  \item[(ii)]
   $P_1$ being invertible ensures that $H^1(a,b;\K^n)$ is the maximal domain for the differential operator $P_1 \frac{\partial}{\partial \z} + P_0$ on $L^2(a,b;\K^n)$, i.e.\ $A$ is a closed operator (otherwise it could never be the generator of a semigroup).
  \item[(iii)]
   $A$ is dissipative if and only if $\Re P_0 \leq 0$ is negative and $W_B \Sigma W_B^* \geq 0$ is positive semi-definite, and in this case $A$ defined below already generates a contractive $C_0$-semigroup on $L^2(a,b;\K^n)$.
  \item[(iv)]
   The assumption on $\H$ ensures that the family of operators $B(t) := \H(t,\cdot)$ acts as matrix multiplication operators on $L^2(a,b;\K^n)$ and satisfies the assumptions of Section \ref{Section1}.
 \end{itemize}
\end{remark}

\subsection{$C^1$-wellposedness and exponential stability}

\begin{theorem}\label{Main Theorem 2}
Let Assumption \ref{Assumption1} hold true.
Then, the non-autonomous port-Hamiltonian system \eqref{Nport-Hamiltonian system}-\eqref{NBoundary control} is $C^1$-well posed with regularity spaces 
\[Y_t=\Big\{x\in L^2(a,b;\K^n)\mid \H(t,\cdot)x\in H^1(a,b;\K^n) \text{ and } \tilde W_B\left[\begin{array}{c}
\H(t,b)x(t,b) \\
\H(t,a)x(t,a) \\
\end{array}\right]=0 \Big\}, \quad t \geq 0\]
Moreover, for each compact interval $[0,T]$ and every classical solution $x$ of \eqref{Nport-Hamiltonian system}-\eqref{NBoundary control} we have
\begin{equation}\label{key formula}
\|x(t)\|_t^2\leq e^{c_T(t-s)}\|x(s)\|_s^2\quad (0\leq s\leq t\leq T),
\end{equation}
for some constant $c_T\geq 0$ that depends only on $m$,  $\underset{t\in[0,T]}{\max}\|\dot \H(t,\cdot)\|$ and $\underset{t \in [0,T]}{\max}\|K(t)\|$.
 If, additionally,
  \begin{equation}
   \H(t,\zeta) K(t,\zeta) + K^\ast(t,\zeta) \H(t, \zeta) + \frac{\partial}{\partial t} \H(t, \zeta)
    \leq 0,
    \quad
    \text{for every } t \geq 0 \text{ and a.e.\ } \z \in [a,b],
    \label{ineq:constraint}
  \end{equation}
 then
  \[
   \|x(t)\|_t
    \leq \|x(s)\|_s,
    \quad
    \text{for } 0 \leq s \leq t
  \]
 holds for every classical solution $x$ of \eqref{Nport-Hamiltonian system}--\eqref{NBoundary control}.
\end{theorem}
 \begin{remark}
  Note that for the special case $K = 0$, condition \eqref{ineq:constraint} just means that $\partial_t \H(t,\zeta) \leq 0$ for every $t \geq 0$ and a.e.\ $\z \in [a,b]$.
 \end{remark}
\begin{proof}[Proof of Theorem \ref{Main Theorem 2}]
The operator $A$ with domain $D(A)$ as defined in \eqref{def:A} generates a contraction $C_0$-semigroup on $X$ by \cite[Theorem 1.1]{Jac-Mor-Zwa15}. The claim then follows from Theorem \ref{main Thm1} and Corollary \ref{cor:passive_phs} since $L^{\infty}(a,b;\K^{n\times n})$ can be seen as a subspace of $\L(X)$ by identifying a function $F\in L^{\infty}(a,b;\K^{n\times n})$ with multiplication operator $x\mapsto Fx$ acting on $X$.
\end{proof}
Under slightly more restrictive regularity conditions, we are able to state the following uniform exponential stability theorem, provided dissipative boundary conditions are imposed.

\begin{theorem}\label{main Theorem Stability}
Let Assumption \ref{Assumption1} hold true, and assume that the evolution family $\mathcal{U} = (U(t,s))_{t \geq s \geq 0}$ generated by $\mathcal{A} = \{A \H(t,\cdot): \, t \geq 0\}$ is contractive, i.e.\ it holds that
 \begin{equation}
  \|U(t,s) x\|_t
   \leq \|x\|_s,
   \quad
   0 \leq s \leq t, \, x \in X.
   \label{ineq:contractivity}
 \end{equation}
In addition, assume that
 \[
  \H \in C^2([0,\infty); C([a,b];\K^n)) \cap C_b^1([0,\infty); C([a,b];\K^n)) \cap L^\infty([0,\infty); \operatorname{Lip}([a,b];\K^n)). 
 \]
Assume that there exists $\kappa> 0$ such that one of the following two conditions holds for all $x \in D(A)$:
 \begin{align}
  \Re \, (A x \mid x)
   &\leq - \kappa |x(b)|^2
   \label{Main stabThmCond1}
   \\
  \Re \, (A x \mid x)
   &\leq - \kappa |x(a)|^2.
   \label{Main stabThmCond2}
 \end{align}
Then the system (\ref{Nport-Hamiltonian system})-(\ref{NBoundary control}) is uniformly exponentially stable, i.e there are constants $\omega<0$ and $L\geq 1$ such that for all classical solutions $x$ of (\ref{Nport-Hamiltonian system})-(\ref{NBoundary control})
\begin{equation}\label{Them stability estimate}
\|x(t)\|\leq L e^{\omega (t-s)}\|x(s)\|,\qquad \text{ for all } 0 \leq s \leq t.
\end{equation}
\end{theorem}
For the proof of this stability theorem, we need the following lemmata. First, we have the following finite observability estimate.
\begin{lemma}[Finite observability estimate] \label{Lemma Stability}Assume that the conditions of Theorem \ref{Main Theorem 2} hold. In addition assume that $\H(t,\cdot)$ is Lipschitz continuous uniformly in time on $[a,b],$ i.e., $\frac{\partial}{\partial \z}\H\in L^\infty([0,\infty)\times [a,b];\K^{n\times n}).$ Then there exist constants $\tau>0$ and $C_\tau>0$ such that for each classical solution $x$ of (\ref{Nport-Hamiltonian system}) we have
\begin{eqnarray}
\label{Eq1: stabilityLemma}\|x(\tau)\|^2_\tau&\leq& C_\tau\int_0^\tau|\H(t,b)x(t,b)|^2 \d t,\\
\label{Eq2: stabilityLemma} \|x(\tau)\|^2_\tau&\leq& C_\tau\int_0^\tau|\H(t,a)x(t,a)|^2 \d t.
\end{eqnarray}
\end{lemma}
\begin{remark}
Note that for the finite observability estimate to hold true, contractivity of the evolution family $\mathcal{U}$ is not necessary, however, the constant $C_\tau$ in the estimates \eqref{Eq1: stabilityLemma} and \eqref{Eq2: stabilityLemma} heavily depends on $\tau$ whenever $\mathcal{U}$ is not bounded. In particular, for the constant $C_\tau$ caclulated in the proof below, one has $C_\tau \rightarrow \infty$ exponentially fast as $\tau \rightarrow \infty$, which prevents the proof of Theorem \ref{main Theorem Stability} from going through.
\end{remark}
\begin{proof}[Proof of Lemma \ref{Lemma Stability}] For the proof, we follow the same strategy as in \cite[Lemma III.1]{Vi-Zw-Go-M}.
Within the proof, we will use that at least on every bounded interval $[0, \tau]$ there is a constant $M_\tau \geq 1$ such that $\|U(t,s) x\|_t \leq M_\tau \|x\|_s$ for all $x \in D(A \H(s))$ and $0 \leq s \leq t \leq T$.
In concrete situations this constant $M_\tau$ may be given by $M_\tau = \mathrm{e}^{\omega \tau}$ (for general port-Hamiltonian systems, thus $M_\tau$ depending exponentially on $\tau$), or $M_\tau = M_0$ indepedent of $\tau$, e.g.\ if the constraint \eqref{ineq:constraint} holds true (in the latter case $M_0 = 1$).
Let $\gamma>0$ and $\tau>0$ be chosen such that $\tau>2\gamma(b-a).$ Let $x$ be a classical solution of  (\ref{Nport-Hamiltonian system}) and define the function $F:[a,b]\lra \R$ via
\[F(\zeta):=\int_{\gamma(b-\zeta)}^{\tau-\gamma(b-\zeta)} (x(t,\zeta) \mid \H(t,\zeta) x(t,\zeta)) \d t, \ \zeta\in[a,b].\]
Note that $F(b) = \int_0^\tau (x(t,b) \mid \H(t,b) x(t,b)) \d t$, and hence
 \[
  \frac{1}{M} \int_0^\tau |\H(t,b) x(t,b)|^2 \d t
   \leq F(b)
   \leq \frac{1}{m} \int_0^\tau |\H(t,b) x(t,b)|^2 \d t.
 \]
For simplicity we sometimes plainly write $x, \H$ instead of $x(t,\zeta), \H(t,\zeta)$.
W.l.o.g.\ we may and will absorb $P_0 \H$ into $K$, i.e.\ we assume that $P_0 = 0$ in the following.
Note that from $\H$ being Lipschitz-continuous it follows that $\H$ and $x$ are weakly differentiable with derivatives in $L^\infty$.
Then we have 
\begin{align*}
\frac{\d}{\d \zeta}F(\zeta)&=\int_{\gamma(b-\zeta)}^{\tau-\gamma(b-\zeta)}\Big[x^*\frac{\partial}{\partial\zeta}(\H x)+(\frac{\partial}{\partial\zeta}x)^*\H x\Big] \d t
\\&\qquad +\gamma (x^*\H x)(\tau-\gamma(b-\zeta),\zeta)
+\gamma (x^*\H x)(\gamma(b-\zeta),\zeta)
\\&= \int_{\gamma(b-\zeta)}^{\tau-\gamma(b-\zeta)}\Big[x^*P_1^{-1}\big(\frac{\partial x}{\partial t} -K x \big)+\Big(P_1^{-1}\frac{\partial}{\partial t}x-\frac{\partial\H}{\partial\zeta}x  - P_1^{-1} K(t,\xi) x \Big)^*x\Big] \d t
\\&\qquad+\gamma (x^*\H x)(\tau-\gamma(b-\zeta),\zeta)
+\gamma (x^*\H x)(\gamma(b-\zeta),\zeta)
\\&= \int_{\gamma(b-\zeta)}^{\tau-\gamma(b-\zeta)} \Big(x^*P_1^{-1}\frac{\partial x}{\partial t}+\frac{\partial x}{\partial t}^*P_1^{-1}x\Big) \d t
\\&\qquad - \int_{\gamma(b-\zeta)}^{\tau-\gamma(b-\zeta)} x^*\Big(K P_1^{-1} + P_1^{-1} K^\ast +\frac{\partial\H}{\partial\zeta}\Big)x \d t 
\\&\qquad+\gamma (x^*\H x)(\tau-\gamma(b-\zeta),\zeta)
+\gamma (x^*\H x)(\gamma(b-\zeta),\zeta).
\end{align*}
Here, we have used that $P_1$ is invertible and self adjoint, $x$ solves (\ref{Nport-Hamiltonian system}) and that $\H$ is self-adjoint as well.  Next, by the fundamental theorem of calculus  we have 
\begin{align*}
\int_{\gamma(b-\zeta)}^{\tau-\gamma(b-\zeta)} \Big(x^*(t,\zeta) P_1^{-1}\frac{\partial x}{\partial t}(t,\zeta) +\frac{\partial x}{\partial t}^*(t,\zeta) P_1^{-1}x(t,\zeta)\Big) \d t&=\Big[x^*(t,\zeta) P_1^{-1}x(t,\zeta)\Big]_{t=\gamma(b-\zeta)}^{\tau-\gamma(b-\zeta)}.
\end{align*}
Therefore,
\begin{align}\label{stabilityLemmaEq11}
\frac{\d}{\d \zeta}F(\zeta)&=- \int_{\gamma(b-\zeta)}^{\tau-\gamma(b-\zeta)} x^*\Big(\H P_0^*P_1^{-1} +P_0P_1^{-1}\H+\frac{\partial\H}{\partial\zeta}\Big)x \d t
\\\label{stabilityLemmaEq12}&\qquad  +\big[x^*(\gamma\H+P_1^{-1}) x\big](\tau-\gamma(b-\zeta),\zeta)+\big[x^*(\gamma\H-P_1^{-1}) x\big](\gamma(b-\zeta),\zeta).
\end{align}
Now, thanks to  \eqref{coercivity H} we can choose $\gamma$ large enough such that 
 \begin{equation}
\label{StabilityLemma Eq1}\pm P_1^{-1}+\gamma \H(t,\xi)\geq 0\qquad (a.e.\, \zeta\in [a,b], t\geq 0).
\end{equation}
Moreover, since $[\gamma(b-\zeta),\tau-\gamma(b-\zeta)]\subset [0,\tau]$ and since $\frac{\partial}{ \partial \zeta}\H\in L^\infty([0,\infty)\times [a,b]; \K^{n\times n})$ there exists $\kappa_\tau>0$ such that 
\begin{equation}\label{StabilityLemma Eq2}
\H(t,\zeta) P_0^*P_1^{-1} +P_0P_1^{-1}\H(t,\zeta)+\frac{\partial\H}{\partial\zeta}(t,\zeta)\leq \kappa_\tau \H(t,\zeta)
\end{equation}
for  a.e  $\zeta\in[a,b]$ and all $t\in [0,\tau].$
For example, we may choose
 \[
  \kappa_\tau
   := 2 \| P_0^* P_1^{-1} \| + \frac{1}{m} \| \frac{\partial \H}{\partial \z} \|_{L^\infty([0,\tau] \times [a,b]; \K^n)}
   \geq 0.
 \]
Inserting \eqref{StabilityLemma Eq1} and \eqref{StabilityLemma Eq2} into \eqref{stabilityLemmaEq11}-\eqref{stabilityLemmaEq12} we obtain that
\[\frac{\d}{\d \zeta}F(\zeta)\geq -\kappa_\tau F(\zeta)\]
holds for  a.e.\  $\zeta\in[a,b]$ and all $t\in [0,\tau].$ This implies that
\begin{equation}\label{StabilityLemma Eq3}
F(\zeta)\leq e^{\kappa_\tau(b-a)}F(b)\quad \text{ for all } \zeta\in[a,b], t\in [0,\tau]
\end{equation}
Using \eqref{key formula}, there exists a constant $c_\tau>0$ (which depends on the interval $[0,\tau]$, more precisely on $\max_{t \in [0,\tau]} \| \dot {\mathcal{H}}(t,\cdot)\|$ and $\frac{1}{m}$) such that 
\begin{align*}
(\tau-2\gamma(b-a))\|\H(\tau-\gamma(b-\zeta,\cdot) x(\tau-\gamma(b-\zeta))\|^2
&\leq e^{ c_\tau \tau}\int_{\gamma(b-\zeta)}^{\tau-\gamma(b-\zeta)}\|x(t)\|^2_t \d t
\\
&=e^{ c_\tau\tau}\int_a^b F(\zeta) \d\zeta
\leq F(b)e^{c_\tau\tau}e^{k_\tau(b-a)}(b-a).
\end{align*}
Here, we have used Fubini's theorem and estimate \eqref{StabilityLemma Eq3} to obtain the last inequality. Taking $\z = b$ and using that $\tau>2\gamma(b-a)$ we conclude 
\[\|x(\tau)\|^2_\tau \leq  F(b) \frac{e^{ c_\tau \tau + \kappa_\tau(b-a)}(b-a)}{\tau-2\gamma(b-a)}= \frac{e^{ c_\tau \tau + \kappa_\tau (b-a)}(b-a)}{\tau-2\gamma(b-a)} \int_0^\tau |\H(t,b) x(t,b)|^2 \d t.\]
This completes the proof of the desired inequality \eqref{Eq1: stabilityLemma} for the constant
 \[
  C_\tau := \frac{e^{ c_\tau \tau + \kappa_\tau(b-a)}(b-a)}{\tau-2\gamma(b-a)}
   > 0.
 \]
The second inequality \eqref{Eq2: stabilityLemma} can be obtained by the same technique.
\end{proof}

\begin{remark}\label{Remark Independance of Constants}
The reasoning in Lemma \ref{Lemma Stability} is by no means restricted to initial time $s=0$, as the constants chosen in the proof lead to inequalities also valid for initial times $s \geq 0$.
Therefore, one obtains the more general finite observability estimates
 \[
  \| x(s + \tau)\|^2_{s+\tau}
   \leq C_\tau \int_s^{s + \tau} |\H(t,b) x(t,b)|^2 \d t,
   \quad s \geq 0
 \]
and
 \[
  \| x(s + \tau)\|^2_{s+\tau}
   \leq C_\tau \int_s^{s + \tau} |\H(t,a) x(t,a)|^2 \d t,
   \quad
   s \geq 0
 \]
respectively, where the constant $C_\tau > 0$ does neither depend on $x$ nor on $s \geq 0$. We will use this property for the proof of Stability Theorem \ref{main Theorem Stability}.
\end{remark}
For the proof of Stability Theorem \ref{main Theorem Stability} one needs to show that there is some energy decay which can be suitably estimated by a time integral over the dissipation terms in conditions \eqref{Main stabThmCond1} and \eqref{Main stabThmCond2}.
In contrast to the situation, one faces the following additional challenges:
   The energy norm $\| \cdot \|_{\H(t)}$ does directly depend on the time-varying Hamiltonian density matrix function $\H(t,\zeta)$, thus is explicitly time-dependent.
   In Theorem \ref{Main Theorem 2} we have seen that, in general, the evolution family $\mathcal{U}$ is neither contractive on $X$ nor does the energy $H(t) = \frac{1}{2} \|U(t,s) x\|_t^2$ ($t \geq s$) decay.
   The best estimate known a priori is only a exponential boundedness estimate; see Theorem \ref{Main Theorem 2}.
   Therefore, the case of exponentially increasing energy somehow has to be excluded right from the start, which can be done by demanding property \eqref{ineq:constraint} which makes the system dissipative in the sense that $H(t)$ decay monotonically, or first showing by other means that the evolution family $\mathcal{U}$ is bounded, and building upon this property.
   As it turns out, however, boundedness of the evolution family still is not enough to deduce exponential stability under the dissipation conditions \eqref{Main stabThmCond1} or \eqref{Main stabThmCond2}; see the counterexample below.  
 As the following corollary shows, for large $\tau > 0$, the constant $C_\tau$ in the observability estimate may be chosen arbitrary small. We state this as a side remark; however, for the proof of Stability Theorem \ref{main Theorem Stability} it actually will not help, but we have to rely on the contraction property \eqref{ineq:contractivity}.
 \begin{corollary}[Small $C_\tau$ for large $\tau$]
  \label{cor:small_c_tau}
  In the situation of Lemma \ref{Lemma Stability}, for every $M_0 > 0$ and $\varepsilon > 0$ there is $\tau > 0$ such that
   \begin{equation}
    \| x(\tau + s) \|_{\tau + s}^2
     \leq C_\tau \int_s^{s + \tau} |\H(t,b) x(t,b)|^2 \d t,
     \quad
    \| x(\tau + s) \|_{\tau + s}^2
     \leq C_\tau \int_s^{s + \tau} |\H(t,a) x(t,a)|^2 \d t
     \label{ineq:bdd-fin-obs-est}
   \end{equation}
  for some $C_\tau \in (0, \varepsilon)$ and every classical solution $x$ of \eqref{Nport-Hamiltonian system} such that $\|x(t+s)\|_{t+s} \leq M_0 \|x(s)\|$ for every $s, t \geq 0$.
 \end{corollary}  
 \begin{proof}
  First, by Lemma \ref{Lemma Stability}, there are $\tau_0 > 0$ and $C_{\tau_0} > 0$ such that
   \[
    \| x(\tau_0 + s) \|_{\tau_0 + s}^2
     \leq C_{\tau_0} \int_s^{s + \tau_0} |\H(t,b) x(t,b)|^2 \d t,
     \quad
    \| x(\tau_0 + s) \|_{\tau_0 + s}^2
     \leq C_{\tau_0} \int_s^{s + \tau_0} |\H(t,a) x(t,a)|^2 \d t
   \]
  for every classical solution $x$ of \eqref{Nport-Hamiltonian system}.
  Now, fix $\tau_0$ and $C_{\tau_0} > 0$ from above and $M_0 > 0$, $\varepsilon > 0$.
  Let $\tau = n \tau_0$ for some $n \in \N$.
  For every classical solution $x$ with $\|x(t+s)\|_{t+s} \leq M_0 \|x(s)\|_s$ ($s,t \geq 0$) one then derives the estimate
   \begin{align*}
    \|x(s+\tau)\|_{s+\tau}^2
     &\leq \frac{M_0^2}{n} \sum_{k=1}^n \|x(s + k \tau_0)\|_{s + k \tau_0}^2
     \\
     &\leq \frac{M_0^2}{n} \sum_{k=1}^n C_{\tau_0} \int_{s + (k-1) \tau_0}^{s + k \tau_0} |\H(t,b) x(t,b)|^2 \d t
     \\
     &= \frac{M_0^2 C_{\tau_0}}{n} \int_s^{s + \tau} |\H(t,b) x(t,b)|^2 \d t,
     \quad
     s \geq 0.
   \end{align*}
  The assertion follows by taking $n > \frac{M_0^2 C_{\tau_0}}{\varepsilon}$ for $\tau = n \tau_0$ and $C_\tau = \frac{M_0^2 C_{\tau_0}}{n} < \varepsilon$.
 \end{proof}
The proof shows that for bounded evolution families, the constant $C_\tau$ can be chosen as $C_\tau \sim \frac{C}{\tau}$ for large $\tau \gg 1$.
More relevant for the proof, resp.\ the formulation of Theorem \ref{main Theorem Stability}, however, is the following remark, characterising contractive evolution families of non-autonomous port-Hamiltonian type.
 \begin{remark}
  The evolution family $\mathcal{U} = (U(t,s))_{t \geq s \geq 0}$ is contractive in the sense that $\|U(t,s) x\|_t \leq \|x\|_s$ for all $0 \leq s \leq t$ and $x \in X$ if and only if $W_B \Sigma W_B^\ast \geq 0$ is positive semi-definite, and condition \eqref{ineq:constraint} holds true.
 \end{remark}
\begin{proof}[Proof of Theorem \ref{main Theorem Stability}]
Let $s \geq 0$ be arbitrary, $x_s \in D(A\H(s))$ and $x: [s, \infty) \rightarrow X$ be the classical solution of \eqref{eqn:NACP} for $B(t) = \H(t,\cdot)$.
By \eqref{key formula}, Theorem \ref{Main Theorem 2} and the assumptions of Theorem \ref{main Theorem Stability} one has
 \begin{equation}\label{uniform estimate} 
  \| x(t+s) \|_{t+s}^2
   \leq \| x(s) \|_s^2,
  \qquad
   s, t \geq 0.
 \end{equation}
According to Lemma \ref{Lemma Stability} and Remark \ref{Remark Independance of Constants} there exists $ \tau >0$ and a constant $C_\tau > 0$ (which does not depend on $s \geq 0$) such that 
\[\|x(\tau+s)\|^2_{\tau+s}\leq C_\tau\int_s^{s+\tau} |\H(t,b)x(t,b)|^2 \d t, \qquad s\geq 0.\]
This inequality together with contractivity of the evolution family implies that (where w.l.o.g.\ we assume that $P_0 = 0$)
\begin{align*}
 &\|x(s+\tau)\|_{s+\tau}^2 - \|x(s)\|_s^2
 =\int_s^{s+\tau} \frac{\d}{\d t} \|x(t)\|_t^2 \d t
\\&=2\Re\int_s^{s+\tau} ( A\H(t)x(t) \mid \H(t)x(t) ) \d t + \int_s^{s+\tau} ( x(t) \mid (K^\ast(t) \H(t) + \H(t) K(t) + \frac{\partial \H(t)}{\partial t}) x(t) ) \d t
\\&\leq - 2\kappa \int_s^{s+\tau} |(\H x)(t,b)|^2 \d t
\leq -\frac{2\kappa}{C_\tau} \| x(s+\tau) \|^2_{s+\tau}
\end{align*}
holds for every  $s \geq 0.$
We deduce that 
\begin{equation}\label{estimate tau}
\|x(s + \tau)\|^2_{s+\tau}
 \leq \rho_\tau \|x(s)\|^2_s
 :=\frac{1}{1+\frac{2\kappa}{C_\tau}} \|x(s)\|^2_s,
 \quad
 s \geq 0.
\end{equation}
 Using \eqref{uniform estimate} and  \eqref{estimate tau} we obtain iteratively for all $s\geq 0$ and  $t=n\tau+r, r\in [0,\tau), n\in\N,$ that 
\begin{align*}
\|x(s+t)\|^2_{t+s}
 &= \|x(s+n\tau+r)\|^2_{s+n\tau+r}
 \leq \|x(s + n\tau)\|^2_{s+n\tau}
 \leq \rho_\tau^n \|x(s)\|^2_s
\\&\leq \rho_\tau^{-1}e^{\frac{\log(\rho_\tau)}{\tau} t } \|x(s)\|^2_s.
\end{align*}
Finally,  according to \eqref{eq LocEquiNorm} and \eqref{coercivity H} we obtain the desired estimate \eqref{Them stability estimate} with \[\omega:=\frac{\log(\rho_\tau)}{\tau}\  \text{and} \ L:= \rho_\tau^{-1}\frac{M}{m}.\]  This completes the proof of the asserted statement.

\end{proof}
\begin{corollary}\label{main Corollary}Assume that the assumptions of Lemma \ref{Lemma Stability} are satisfied. In addition we assume $W_B\Sigma W_B^*>0.$ Then the classical solution $x$ of (\ref{Nport-Hamiltonian system})-(\ref{NBoundary control})  is uniformly exponentially stable.
\end{corollary}
\begin{proof} Since $W_B\Sigma W_B^*>0$, both conditions \eqref{Main stabThmCond1} and \eqref{Main stabThmCond2} hold, e.g.\ by the proof of \cite[Lemma 9.1.4]{Jac-Zwa12}. Now the claim follows from Theorem \ref{main Theorem Stability}.
\end{proof}
\begin{remark}The previous results have been proved in \cite[Theorem III.2]{Vi-Zw-Go-M} and \cite[Theorem 9.1.3, Theorem 7.2.4]{Jac-Zwa12} in the case where $\H$ is independent of $t.$\end{remark}

\begin{remark}
   It would be also possible to consider port-Hamiltonian systems of higher order $N \geq 2$, i.e.\
    \[
     A
      = \sum_{k=0}^N P_k \frac{\partial^k}{\partial \z^k}
    \]
   on an appropriate domain $D(A) \subseteq H^N(a,b;\K^N)$ (including, say, dissipative boundary conditions), where now the conditions on the matrices $P_k$ read: $P_k \in \K^{n \times n}$ with $P_k = (-1)^{k+1} P_k$ for $k \geq 1$ and $P_N$ invertible.
   The $C^1$-well posedness result Theorem \ref{Main Theorem 2} directly transfers to that situation.
   However, a final observability estimate as in Lemma \eqref{Lemma Stability} is not (yet) known for that situation, and proofs for uniform exponential stability in the autonomous situation rather rely on particular semigroup techniques (the ABLV-Theorem and the stability theorem of Gearhart, Pr\"uss and Huang) which are not at hand for non-autonomous problems.
\end{remark}
\subsection{A counterexample on stability}
In this subsection, we demonstrate that, in general, without contractivity of the evolution family as demanded in Theorem \ref{main Theorem Stability}, one may not deduce exponential stability for the evolution family $\mathcal{U}$ from the dissipation conditions \eqref{Main stabThmCond1} or \eqref{Main stabThmCond2}. In fact, neither of the following properties does help in general:
  \begin{enumerate}
   \item
    boundedness of the evolution family, i.e.\ there is $M_0 \geq 1$ such that for all $0 \leq s \leq t$ and $x \in X$: $\|U(t,s) x\|_t \leq M_0 \|x\|_s$.
   \item
    time-periodicity of $\H(t,\cdot)$, i.e.\ there is $t_0 > 0$ such that $\H(t+t_0,\cdot) = \H(t,\cdot)$ for all $t \geq 0$. 
  \end{enumerate}
To show this negative result, let us consider the following example of a two-component system of non-autonomous transport equations.
 \begin{example}
  Consider the following system of transport equations
   \[
    \frac{\partial}{\partial t} x_j(t,\zeta) = h_j(t) \frac{\partial}{\partial \zeta} x_j(t,\zeta),
     \quad
     t \geq 0, \, \zeta \in (0,1), \, j = 1,2
   \]
  with boundary conditions which interconnect the two transmission lines by
   \[
    h_2(t) x_2(t,1)
     = h_1(t) x_1(t,0),
     \quad
    h_1(t) x_1(t,1)
     = \alpha h_2(t) x_2(t,0)
   \]
  for some $\alpha \in \K$.
 \end{example}
 In the following, we investigate well-posedness and stability properties of this system.
 \begin{lemma}
  Let $h_1 = h_2 \equiv 1$, then the system of PDE's
   \[
    \begin{cases}
     \frac{\partial}{\partial t} x_i(t,\z)
      = \frac{\partial}{\partial \z} x_i(t,\z),
      &t \geq 0, \, \z \in (0,1), \, i = 1,2,
      \\
     x_2(t,1)
      = x_1(t,0),
      &t \geq 0,
      \\
     x_1(t,1)
      = \alpha x_2(t,0),
      &t \geq 0,
      \\
     (x_1,x_2)(0,\zeta)
      = (x_{0,1}, x_{0,2})(\zeta),
      &\zeta \in (0,1)
    \end{cases}
   \]
  has for every $\alpha \in \R$ and every initial data
   \[
    x_0 \in D(A_\alpha)
     = \{x \in H^1(0,1;\K^2): x_2(1) = x_1(0), \, x_1(1) = \alpha x_2(0) \}
   \]
  a unique classical solution $x \in C(\R_+; H^1(0,1;\K^2)) \cap C^1(\R_+; L^2(0,1; \K^2))$ which are given by $x(t,\cdot) = T(t) x_0$ for some $C_0$-semigroup $(T(t))_{t \geq 0}$ on $X = L^2(0,1;\K^2)$.
   Moreover, the semigroup is uniformly exponentially stable if and only if $|\alpha| < 1$, and it is isometric if and only if $|\alpha| = 1$.
 \end{lemma} 
 \begin{proof}
  In this particular situation, and using well-known results on the transport equation, one may explicitly write down the (unique) classical solution, for any $\alpha \in \K$.
  The stability property then directly follows from the solution formula.
 \end{proof}
 \begin{remark}
  A port-Hamiltonian formulation of the example above is the following:
   \[
    \frac{\partial}{\partial t} x
     = A x
     = \left[ \begin{array}{cc} 1 & 0 \\ 0 & -1 \end{array} \right] \frac{\partial}{\partial \zeta} x,
     \quad
    D(A)
     = \{x \in H^1(0,1;\K^2): x_2(0) = x_1(0), \, x_1(1) = \alpha x_2(1) \},
   \]
  and it satisfies the dissipation condition
   \[
    \Re \, (A x \mid x)_{L^2}
     = (|\alpha|^2 - 1) |x_2(1)|^2
   \]
  which is less or equal $- \kappa |x(1)|^2$ (for some $\kappa > 0$) in case that $|\alpha| < 1$, for all $x \in D(A)$.
  (Note that we took $x = (\tilde x_1, \tilde x_2)$ for $\tilde x_1 = x_1$ and $\tilde x_2(t,\zeta) := x_2(t,1-\zeta)$ for the port-Hamiltonian formulation.
  In particular, condition \eqref{Main stabThmCond1} for $b = 1$ is satisfied then.
 \end{remark}
 \begin{example}
 \label{exa:counter-exa}
  Next, consider the following time-periodic (but not regular in time) choice of the weight functions $h_1(t), h_2(t)$:
   \[
    h_1(t) = \begin{cases} 2, & t \in [0, \tfrac{1}{2}) + \N_0, \\ 1, &\text{else} \end{cases},
     \quad
    h_2(t)
     = 3 - a_1(t) = \begin{cases} 1, & t \in [0, \tfrac{1}{2}) + \N_0, \\ 2, &\text{else} \end{cases}.
   \]
  Although $\H(t) := \operatorname{diag}\, (h_1(t), h_2(t))$ is not even continuous in time, one may nevertheless associate a evolution family to this problem by setting
   \[
    U(t,s) = \begin{cases} T_0(t-s), &s,t \in [0,\tfrac{1}{2}] + k \quad \text{for some } k \in \N_0 \\ T_1(t-s), &s,t \in [\frac{1}{2}, 1] + k \quad \text{for some } k \in \N_0 \end{cases},
    \quad
    0 \leq s \leq t
   \]
  and glueing $\mathcal{U}$ together, e.g.\
   \[
    U(t,s)
     = T_1(t-l) T_0(\frac{1}{2}) T_1(\frac{1}{2}) \cdots T_0(\frac{1}{2}) T_1(l+\frac{1}{2}-s),
     \quad
     \text{for } t \in l + [0,\frac{1}{2}], \, s \in k + [0, \frac{1}{2}], \, k < l \in \N,
   \]
  to satisfy the evolution family property, where $(T_0(t))_{t\geq0}$ is the $C_0$-semigroup generated by $A \H(0)$ and $(T_1(t))_{t\geq0}$ is the $C_0$-semigroup generated by $A \H(1)$, where we fixed some $\alpha \in \R$ in the boundary conditions.
  We claim that the resulting evolution family is exponentially bounded, but bounded and not exponentially stable for $|\alpha| = \frac{1}{2}$, and unbounded for $|\alpha| > \frac{1}{2}$.
 \end{example}
 \begin{proof}
  Since $(T_0(t))_{t\geq0}$ and $(T_1(t))_{t\geq0}$ are $C_0$-semigroups, there are $M \geq 1$ and $\omega \in \R$ such that $\|T_0(t)\|, \|T_1(t)\| \leq M \mathrm{e}^{\omega t}$ for all $t \geq 0$.
  From the construction of the evolution family, it then follows that $\|U(t,s)\| \leq \tilde{M} \mathrm{e}^{\tilde{\omega} (t-s)}$ for some $\tilde M \geq 1$ and $\tilde \omega \in \R$, and all $0 \leq s \leq t$.
  E.g.\ one may choose $\tilde M = M^2$ and $\tilde \omega = \omega + \ln M$:
   \begin{align*}
    \|U(t,s)\|
     &\leq \|U(t, \frac{\lfloor 2t \rfloor}{2})\| \|U(\frac{\lfloor 2t \rfloor}{2}, \frac{\lfloor 2t \rfloor - 1}{2})\| \cdots \|U(\frac{\lceil 2s \rceil + 1}{2}, \frac{\lceil 2s \rceil}{2})\| \|U(\frac{\lceil 2s \rceil}{2}, s)\|
     \\
     &\leq M^2 M^{t-s} \mathrm{e}^{\omega (t-s)}
     = M^2 \mathrm{e}^{(\omega + \ln M)(t-s)}
     = \tilde M \mathrm{e}^{\tilde \omega (t-s)},
     \quad
     t \geq s \geq 0.
   \end{align*}   
  \newline
  To show that $\mathcal{U}$ is not exponentially stable for $|\alpha| \geq \frac{1}{2}$ consider the initial datum $x_0 \equiv (0,1)^T \in L_2(0,1;\K^2)$.
  One may then calculate $U(n,0) x_0$ for $n \in \N$ as $U(t,0) x_0 \equiv (2 \alpha)^n (1,0)^T \in L^2(0,1;\K^2)$.
  Hence, for $|\alpha| = \tfrac{1}{2}$, the evolution family is not exponentially stable, and for $|\alpha| > \frac{1}{2}$ it is unstable.
  Moreover, for $|\alpha| = \frac{1}{2}$, using the construction of the evolution family, one may show that $|(U(n,0) x_0)(\z)| \leq |x_0(\z)|$ for all $n \in \N_0$ and a.e.\ $\z \in [0,1]$, in particular, the maps $U(n,0) \in \L(X)$ are contractions. From the exponential boundedness it then follows for all $0 \leq s \leq t$ that
   \[
    \|U(t,s)\|
     \leq \|U(t,\lfloor t \rfloor)\| \|U(\lfloor t \rfloor, \lceil s \rceil)\| \|U(\lceil s \rceil, s)\|
     = \|U(t,\lfloor t \rfloor)\| \|U(\lfloor t \rfloor - \lceil s \rceil, 0)\| \|U(\lceil s \rceil, s)\|
     \leq M^2 \mathrm{e}^{2 |\omega|}
   \]
  whenever $\lfloor t \rfloor \geq \lceil s \rceil$. The case $\lfloor t \rfloor < \lceil s \rceil$ is even easier and one then gets $\|U(t,s)\| \leq M \mathrm{e}^{|\omega|}$.
 \end{proof}
 \begin{remark}
  Clearly, Example \ref{exa:counter-exa} does not have the time-regularity of $\H$ demanded in Theorem \ref{main Theorem Stability}, however, one might easily adjust the example to have $a_1, a_2$ in class $C^2$, so that one cannot give up on the condition of contractivity in Theorem \ref{main Theorem Stability}; even time-periodicity of the Hamiltonian matrix functions does not help then.
 \end{remark}
\section{Examples }
In this section, the abstract results of the preceeding section are applied to some particular examples of beam equation; namely, the one-dimensional linear wave equation as a model for a rather flexible beam or string, e.g.\ of a musical instrument, and secondly the Timshenko beam model which is more detailed and better suitable for rigid beams, as it also includes a rotational displacement from the equilibrium configuration as a variable.

\subsection{A non-autonomous vibrating string}
As a first example, let us consider the model of a vibrating string on the compact interval $[a,b]$. We assume that the string is fixed at the left end point $\z = a$ and at the right end point $\z = b$ a linear damper is attached. In addition, Young's modulus $T$ and the mass density $\rho$ of the string are assumed to be time- and spatial dependent. Let us denote by $\omega(t,\zeta)$ the vertical position of the string at position $\zeta\in[a,b]$ and time $t\geq 0.$ Then the evolution of the vibrating string can be modelled by the non-autonomous wave equation
\begin{align}\label{eq: Vibrating string} 
\frac{\partial  }{\partial t}\Big(\rho(t,\zeta)\frac{\partial w }{\partial t}(t,\zeta) \Big)&=\frac{\partial}{\partial \zeta}\Big(T(t,\zeta)\frac{\partial w}{\partial \zeta}(t,\zeta)\Big), && \zeta\in[a,b],\ t\geq 0,\\
\label{Eq1: BCd}\frac{\partial w}{\partial t}(t,a)&=0, && t\geq 0
\\ \label{Eq2: BCd}T(t,b)\frac{\partial w}{\partial \zeta}(t,b)&=-k\frac{\partial w}{\partial t}(t,b), && t\geq 0.
\end{align} 
Taking the momentum-strain couple $(\rho \frac{\partial w}{\partial t}, \frac{\partial w}{\partial \zeta})$ as the energy state space variable, one sees that the non-autonomous vibrating string can be written as a system of the form  \eqref{Nport-Hamiltonian system}-\eqref{NBoundary control} with $P_0=0,$ \[P_1=\begin{bmatrix} 
 0&1\\
 1&0
 \end{bmatrix},\quad\H(t,\zeta)=\begin{bmatrix} 
 \frac{1}{\rho(t,\zeta)}&0\\
 0&T(t,\zeta)\end{bmatrix} \text{ and } \tilde W_B= \begin{bmatrix} k&1&0&0 \\ 0&0&1&0 \end{bmatrix}.\]
We assume that the Young's modulus function $T$ and  the mass density function $\rho$ are measurable and satisfy the following regularity and positivity conditions:
\begin{itemize}
\item [$(i)$] $T, \rho \in C^2([0,\infty);L^\infty(a,b)) \cap C_b([0,\infty);L^\infty(a,b))$
\item [$(ii)$] There is a constant $\alpha>0$ such that for a.e.\  $\zeta\in[a,b]$ and all $t\geq 0$
\begin{equation}\label{Exp1: Assumption 1}
\alpha^{-1} \leq \rho(t,\zeta), T(t,\zeta)\leq \alpha.
\end{equation} 
\end{itemize}

\noindent 
Recall that \eqref{eq: Vibrating string}-\eqref{Eq2: BCd} can be  reformulated  in the port-Hamiltonian form by choosing  the momentum-strain couple $(\rho \frac{\partial w}{\partial t}, \frac{\partial w}{\partial \zeta})$, i.e.\ the energy variables, as the  state variable with $P_0=0,$
\[P_1=\begin{bmatrix} 
0&1\\
1&0
\end{bmatrix} \quad \text{ and } \quad \H(t,\zeta)=\begin{bmatrix} 
\frac{1}{\rho(t,\zeta)}&0\\
0&T(t,\zeta)
\end{bmatrix}.\]
\noindent Moreover, the boundary conditions \eqref{Eq1: BCd}-\eqref{Eq2: BCd} can be reformulated as follows 
\begin{align*}\begin{bmatrix} 
0\\
0
\end{bmatrix}&=
\begin{bmatrix} k&1&0&0\\
0&0&1&0\end{bmatrix}\begin{bmatrix} \H(t,b)x\\ \H(t,a)x\end{bmatrix}
=:\tilde W_B\begin{bmatrix} \H(t,b)x\\ \H(t,a)x\end{bmatrix}
\end{align*}
The $2\times 4$ matrix \[W_B=\tilde W_B \begin{bmatrix}
P_1& -P_1\\
I&I
\end{bmatrix}^{-1}=
\begin{bmatrix} 1&k&k&1\\
0&-1&1&0\end{bmatrix}\]
has rank $2$ and $W_B\Sigma W_B^*=\begin{bmatrix}
2k& 0\\
0&0
\end{bmatrix}\geq 0.$
Integration by parts gives that
 \begin{equation}
  2 \Re \, (A x \mid x)
   = \left[ (x(b) \mid P_1 x(b) ) - (x(a) \mid P_1 x(a)) \right]
 \end{equation}
holds for all $x\in D(A)$.
By this equality, together with the boundary conditions \eqref{Eq1: BCd}-\eqref{Eq2: BCd}, we observe that for $x=\begin{bmatrix} 
\rho\frac{\partial w}{\partial \zeta}\\
\frac{\partial w}{\partial \zeta}
\end{bmatrix}$
\[2\Re (Ax\mid x) \leq -\frac{k}{1+k^2} |\H(b,t)x(b)|^2,\quad x \in D(A) = \{x \in H^1(a,b; \K^2): \, x_1(a) = 0, \, x_2(b) = - k x_1(b)\}.\]

\medskip

\noindent Thus the following well-posedness and stability results follows from Theorem \ref{main Theorem Stability}.
\begin{proposition} Let $\omega_0\in H^1(a,b)$ and $\omega_1 \in H^1(a,b)$ be such that $T(0,\cdot) \frac{\partial \omega_0}{\partial \z} \in H^1(a,b)$ and $T(b,0)\frac{\partial w_0}{\partial \zeta}(b)=0$.
Moreover, assume that $\omega_1(a) = 0$ and $T(b,0) \omega_0'(b) = - k \omega_0(b)$.
Then, \eqref{eq: Vibrating string} with boundary conditions \eqref{Eq1: BCd}-\eqref{Eq2: BCd} and initial conditions
 \[
  \omega(0,\cdot) = \omega_0,
   \quad
   \frac{\partial \omega}{\partial t}(0,\cdot) = \omega_1
 \]
has a unique solution $\omega$ such that \[\left(t\longmapsto \begin{bmatrix}\frac{\partial w}{\partial t}\\
T(t,\cdot) \frac{\partial w_0(t,\cdot)}{\partial \zeta}\end{bmatrix}\right) \in C^1\big([0,\infty); L^2(a,b;\K^2)\big)\cap C\big([0,\infty); H^1(a,b;\K^2)\big) \]
If, in addition, $T, \rho \in C^2([0,\infty);C([a,b])) \cap C_b^1([0,\infty);C([a,b])) \cap L^\infty([0,\infty); \operatorname{Lip}([a,b]))$ and $T, \rho^{-1}$ are decreasing with respect to the time variable, then we have
\begin{equation}
\|\frac{\partial w}{\partial t}(t,\cdot)\|_{L^2(a,b)}^2+\|T(t,\cdot)\frac{\partial w}{\partial \zeta}(t,\cdot)\|_{L^2(a,b)}^2
 \leq Me^{\omega t} \left( \|\frac{\partial w_0}{\partial \zeta}(\cdot)\|_{L^2(a,b)}^2 + \| \omega_1 \|_{L_2(a,b)}^2 \right)
\end{equation}
for all $t\geq 0$ and some constants $M\geq 1$ and $\omega < 0$ that are independent of $t \geq 0$ and the initial data $\omega_0, \omega_1$.
More precisely, for all $s, t \geq 0$ one has the estimate
 \[
  \|\frac{\partial w}{\partial t}(t+s,\cdot)\|_{L^2(a,b)}^2+\|T(t+s,\cdot)\frac{\partial w}{\partial \zeta}(t+s,\cdot)\|_{L^2(a,b)}^2
 \leq Me^{\omega t} \left( \|\frac{\partial w}{\partial t}(s,\cdot)\|_{L^2(a,b)}^2+\|T(t,\cdot)\frac{\partial w}{\partial \zeta}(s,\cdot)\|_{L^2(a,b)}^2 \right).
 \]
\end{proposition}
\subsection{Non-autonomous Timoshenko Beam}  
As a second example, consider the non-autonomous version of the Timoshenko beam model given by the equations:
\begin{align}\label{TimoshenkoEq1}
\frac{\partial }{\partial t}\big(\rho(t,\zeta)\frac{\partial w}{\partial t}(t,\zeta)\big)&=\frac{\partial}{\partial \zeta}\Big[ K(t,\zeta)\Big(\frac{\partial}{\partial \zeta}w(t,\zeta)+\phi(t,\zeta)\Big)\Big]
\\
\label{TimoshenkoEq2}\frac{\partial }{\partial t}\big(I_\rho(t,\zeta)\frac{\partial \phi}{\partial t}(t,\zeta)\big)&=\frac{\partial}{\partial \zeta}\Big(EI(t,\zeta)\frac{\partial ^2}{\partial \zeta}\phi(t,\zeta)\Big)+K(t,\zeta)\Big(\frac{\partial }{\partial \zeta}w(t,\zeta)-\phi(t,\zeta)\Big)
\end{align}
where $\zeta \in (a,b), t\geq 0,$ $w(t,\zeta)$ is the transverse displacement of the beam and $\phi(t,\zeta)$ is the rotation angle of the filament of the beam. Dropping the coordinates $\zeta$ and $t$ in the notation and taking $x:=(\frac{\partial w}{\partial \zeta}-\phi, \rho\frac{\partial w}{\partial t}, \frac{\partial \phi}{\partial \zeta}, I_\rho\frac{\partial\phi}{\partial t})$ as state variable, one can see that the Timoshenko beam model may be written as a system of the form (\ref{Nport-Hamiltonian system})-(\ref{NBoundary control}) with 
\medskip

\[P_1=\begin{bmatrix}
0 & 1 & 0 & 0\\
1 & 0 & 0 & 0\\
0 & 0 & 0 & 1\\
0 & 0 & 1 & 0\\
\end{bmatrix},
\quad
P_0=\begin{bmatrix}
0 & 0 & 0 & -1\\
0 & 0 & 0 & 0\\
0 & 0 & 0 & 0\\
1 & 0 & 0 & 0\\
\end{bmatrix}
\quad \text{and} \quad
\H=\begin{bmatrix}
K & 0 & 0 & 0\\
	0 & \rho^{-1} & 0 & 0\\
	0 & 0 & EI & 0\\
0 & 0 & 0 & I_\rho^{-1}\\
\end{bmatrix}.\]
\medskip

At the right and left end of the Timoshenko beam, we impose the following boundary conditions
\begin{align}\frac{\partial w}{\partial t}(t,a)&=0,\qquad &t\geq 0
\\\frac{\partial\phi}{\partial t}(t,a)&=0,\qquad  &t\geq 0
\\ K(t,b)\Big[\frac{\partial w}{\partial \zeta}(t,b)-\phi(t,b)\Big]&=-\alpha_1\frac{\partial w}{\partial t}(t,b), \qquad  &t\geq 0
\\ EI(t,b)\frac{\partial \phi}{\partial \zeta}(t,b)&=-\alpha_2\frac{\partial\phi}{\partial t}(t,b),\qquad  &t\geq 0
\end{align}
for some positive constants $\alpha_1, \alpha_2\geq 0$, i.e.\ we impose conservative boundary conditions at the left end $\z = a$ and dissipative feedback at the right end $\z = b$ of the beam.
These boundary conditions can be written as 
\begin{align*}\begin{bmatrix} 
0\\
0\\
0\\
0
\end{bmatrix}&=\begin{bmatrix} 0&0&0&0&0&1&0&0
\\ 0&0&0&0&0&0&0&1
\\ 1&\alpha_1&0&0&0&0&0&0
\\ 0&0&1&\alpha_2&0&0&0&0
\end{bmatrix}\begin{bmatrix} \H(t,b)x\\ \H(t,a)x\end{bmatrix} :=\tilde W_B \begin{bmatrix} \H(t,b)x\\ \H(t,a)x\end{bmatrix}
\end{align*}
The corresponding $4\times 8$ matrix $W_B$ is given by 
\[W_B= \tilde W_B \begin{bmatrix}
P_1& -P_1\\
I&I
\end{bmatrix}^{-1}=\frac{1}{2} \begin{bmatrix} -1&0&0&0&0&1&0&0
\\ 0&0&-1&0&0&0&0&1
\\\alpha_1&1&0&0&1&\alpha_1&0&0
\\0&0&\alpha_2&1&0&0&1&\alpha_2
\end{bmatrix}\]
has full rank and $W_B\Sigma W_B^*= \begin{bmatrix}0&0&0&0
\\ 0&0&0&0
\\ 0&0&\alpha_1&0
\\ 0&0&0&\alpha_2
\end{bmatrix}\geq 0.$ 
\par 

We conclude that the non-autonomous Cauchy problem associated with the Timoshenko beam \eqref{TimoshenkoEq1}-\eqref{TimoshenkoEq1} is $C^1$-well posed provided the physical parameters defining the Hamiltonian matrix $\H$ satisfiy similar conditions to those described for the vibrating string, i.e.\ $K, \rho, EI, I_\rho \in C^2([0,\infty);L^\infty(a,b)) \cap C_b^1([0,\infty);L^\infty(a,b))$. Moreover, by Theorem \ref{main Theorem Stability} the associated evolution family is uniformly exponentially stable as long as $\alpha_1, \alpha_2>0$ are both strictly positive and $K, \rho, EI, I_\rho \in C^2([0,\infty);C([a,b])) \cap C_b^1([0,\infty);C([a,b])) \cap L^\infty([0,\infty);\operatorname{Lip}([a,b])$ and $K, \rho^{-1}, EI, I_\rho^{-1}$ are decreasing with respect to the time variable $t$.

\end{document}